\documentclass[a4paper,11pt]{amsart}

\usepackage[hmargin=1in,vmargin=1.1in]{geometry}
\usepackage[utf8]{inputenc}
\usepackage[T1]{fontenc}
\usepackage{lmodern,microtype,enumitem,mathtools,thmtools,amssymb,eucal,tikz,hyperref}
\usepackage[noadjust]{cite}
\usepackage[compress,sort,nameinlink]{cleveref}
\usetikzlibrary{decorations.pathreplacing}

\hypersetup{
  pdfauthor={Paweł Rzążewski, Bartosz Walczak},
  pdftitle={Polynomial-time recognition and maximum independent set in Burling graphs},
  colorlinks,
  linkcolor={red!50!black},
  citecolor={blue!50!black},
  urlcolor={blue!80!black}
}

\makeatletter
\let\oldthanks\thanks
\def\thanks#1{\oldthanks{\hskip-\parindent #1}}
\def\@setaddresses{%
  \par\nobreak\begingroup\footnotesize\interlinepenalty\@M\bigskip
  \def\author##1{}%
  \def\address##1##2{%
    \par\addvspace\bigskipamount\noindent
    \@ifnotempty{##1}{\textit{##1}: }{\ignorespaces##2}%
  }%
  \def\email##1##2{\ifvmode\else\unskip; \fi\textit{email}: ##2}%
  \addresses\par
  \endgroup
}
\def\thm@setnote#1{(#1)}
\def\thmhead@plain#1#2#3{%
  \thmname{#1}\thmnumber{\@ifnotempty{#1}{ }\@upn{#2}}%
  \thmnote{ {\the\thm@notefont\thm@setnote{#3}}}%
}
\let\thmhead\thmhead@plain
\def\th@algostyle{\def\thm@setnote##1{##1}}
\def\nobreakpar{\par\nobreak\@afterheading}
\makeatother

\newlist{conditions}{enumerate}{1}
\newlist{properties}{enumerate}{1}
\newlist{axioms}{enumerate}{1}
\newlist{steps}{enumerate}{1}
\newlist{options}{enumerate}{1}
\setlist[itemize]{label=\textbullet, leftmargin=\widthof{1.\enspace}}
\setlist[enumerate,conditions,properties,1]{label=\upshape{(\arabic*)}, leftmargin=*}
\setlist[enumerate,2]{label=\upshape{(\arabic{enumi}.\arabic*)}, leftmargin=*}
\setlist[axioms]{label=\upshape{(\roman*)}, leftmargin=*, widest=iii}
\setlist[steps,options]{label=\upshape{\arabic*.}, ref=\upshape{\arabic*}, leftmargin=*}

\crefname{conditionsi}{condition}{conditions}
\crefname{propertiesi}{property}{properties}
\crefname{axiomsi}{axiom}{axioms}
\crefname{stepsi}{step}{steps}

\newtheorem{theorem}{Theorem}
\newtheorem{lemma}[theorem]{Lemma}
\theoremstyle{definition}
\newtheorem{problem}{Problem}
\theoremstyle{algostyle}
\newtheorem*{algoenv}{Algorithm to solve}
\newenvironment{algorithm}[1]{\algoenv[#1]\leavevmode\nobreakpar\smallskip\steps}{\endsteps\endalgoenv}

\newcommand\setR{\mathbb{R}}
\newcommand\famC{\mathcal{C}}
\newcommand\famF{\mathcal{F}}

\newcommand\weight{\mathsf{w}}
\newcommand\Rooted{\mathfrak{R}}
\newcommand\Unrooted{\mathfrak{B}}
\newcommand\Indep{\mathfrak{I}}
\newcommand\adj{\curvearrowright}
\newcommand\rel{\mathrel{R}}
\newcommand\edgeset[1]{E_{\smash{#1}}}
\newcommand\alignto[3]{\makebox[\widthof{$#1$}][#2]{$#3$}}
\newcommand\Oh{\mathcal{O}}
\newcommand\classP{\textup{\sf P}}
\newcommand\classNP{\textup{\sf NP}}

\DeclarePairedDelimiter\set{\{}{\}}
\DeclarePairedDelimiter\size{\lvert}{\rvert}

\let\leq\leqslant
\let\geq\geqslant
\let\epsilon\varepsilon

\begin{document}

\title{Polynomial-time recognition and maximum independent~set in Burling graphs}
\author{Paweł Rzążewski\and Bartosz Walczak}

\address[Paweł Rzążewski]{Faculty of Mathematics and Information Science, Warsaw University of Technology, Warsaw, Poland, and Institute of Informatics, University of Warsaw, Warsaw, Poland}
\email{pawel.rzazewski@pw.edu.pl}
\address[Bartosz Walczak]{Department of Theoretical Computer Science, Faculty of Mathematics and Computer Science, Jagiellonian University, Kraków, Poland}
\email{bartosz.walczak@uj.edu.pl}

\thanks{A preliminary version of this paper appeared in the proceedings of the 51st International Workshop on Graph-Theoretic Concepts in Computer Science (WG 2025)~\cite{RW25}.\\
Paweł Rzążewski was partially supported by the National Science Center of Poland grant 2024/54/E/ST6/00094.\\
Bartosz Walczak was partially supported by the National Science Center of Poland grant 2019/34/E/ST6/00443.}

\begin{abstract}
A Burling graph is an induced subgraph of some graph in Burling's construction of triangle-free high-chromatic graphs.
Equivalently, a Burling graph is a graph that admits a so-called strict frame representation.
We provide a polynomial-time algorithm to decide whether a given graph is a Burling graph and if it is, to construct its strict frame representation.
The representation then enables a polynomial-time algorithm for the maximum independent set problem in Burling graphs.
As a consequence, we establish Burling graphs as the first known hereditary class of graphs that admits such an algorithm while not being $\chi$-bounded.
\end{abstract}

\maketitle

\section{Introduction}

A class of graphs is \emph{$\chi$-bounded} if the chromatic number of the graphs in the class is bounded by some function of the clique number.
In 1965, Burling~\cite{Bur65} constructed an infinite sequence $B_1,B_2,\ldots$ of triangle-free graphs with $\chi(B_k)=k$ that admit an intersection representation by axis-parallel boxes in $\setR^3$, showing that the class of graphs with such representations is not $\chi$-bounded.
This answered a question by Asplund and Grünbaum~\cite{AG60}, who proved in 1960 that the class of graphs with intersection representations by axis-parallel rectangles in $\setR^2$ is $\chi$-bounded.
The importance of Burling's construction has been realized only in the last decade, since Pawlik et~al.\ \cite{PKK+13,PKK+14} devised its intersection representations by various geometric shapes in the plane, such as straight-line segments or axis-parallel rectangular frames.
Such representations imply that Burling's construction excludes induced subdivisions of $1$-subdivided non-planar graphs.
As a consequence, they disprove a conjecture of Scott~\cite{Sco97} that every class of graphs excluding induced subdivisions of a fixed graph is $\chi$-bounded.
This motivated further study of what graphs occur as induced subgraphs in Burling's construction \cite{CELO16,Dav23,PT24a,PT24b}.
Following Pournajafi and Trotignon \cite{PT23,PT24a,PT24b}, we call a graph a \emph{Burling graph} if it is an induced subgraph of $B_k$ for some $k$.

Burling graphs can be characterized in terms of so-called strict frame representations.
A~\emph{frame representation} of a graph $G$ assigns an axis-parallel rectangular frame to every vertex of $G$ so that the edges of $G$ comprise the pairs of vertices whose frames intersect.
Note that a frame is an empty rectangle (excluding the interior), so nested frames do not intersect.
A frame representation is \emph{strict} if the left side of every frame intersects no other frame and lies inside no two intersecting frames (see \Cref{fig:strict}).
It follows from the construction in~\cite{PKK+13} that every Burling graph has a strict frame representation, and Pournajafi and Trotignon~\cite{PT23} proved that a graph is a Burling graph if and only if it admits a strict frame representation.
Various aspects of graphs with frame representations (a.k.a.\ rectangle overlap graphs) were also studied in \cite{CELO16,KPW15,KW17}.

Pournajafi and Trotignon~\cite{PT24a} asked about the complexity of recognizing Burling graphs.
We answer their question by proving that the problem belongs to the complexity class~\classP.\@

\begin{theorem}
\label{thm:recognition-frames}
There is a polynomial-time algorithm that, given a graph\/ $G$, decides whether\/ $G$ is a Burling graph and if it is, produces a strict frame representation of\/ $G$.
\end{theorem}

\begin{figure}[t]
\begin{center}
\tikzset{every picture/.style={xscale=0.75,yscale=0.5}}
\tikzset{every node/.style={below,yshift=-5pt}}
\begin{tikzpicture}
  \draw (0,0) rectangle (2,3);
  \draw (0.5,0.5) rectangle (3,2.5);
  \node at (1.5,0) {(a)};
\end{tikzpicture}\hskip 1cm
\begin{tikzpicture}
  \draw (0,0) rectangle (2,3);
  \draw (0.5,0.5) rectangle (3,2.5);
  \draw (1,1) rectangle (1.5,2);
  \node at (1.5,0) {(b)};
\end{tikzpicture}\hskip 1cm
\begin{tikzpicture}
  \draw (0,0) rectangle (2,3);
  \draw (0.5,0.5) rectangle (3,2.5);
  \draw (1,1) rectangle (2.5,2);
  \node at (1.5,0) {(c)};
\end{tikzpicture}\hskip 1cm
\begin{tikzpicture}
  \draw (0,0) rectangle (2,3);
  \draw (0.5,0.5) rectangle (3,2.5);
  \draw (1,1) rectangle (4,2);
  \node at (1.85,0) {(d)};
\end{tikzpicture}
\end{center}
\caption{Strict family of frames: (a) required configuration of every intersecting pair of frames; (b)--(d) disallowed configurations.}
\label{fig:strict}
\end{figure}
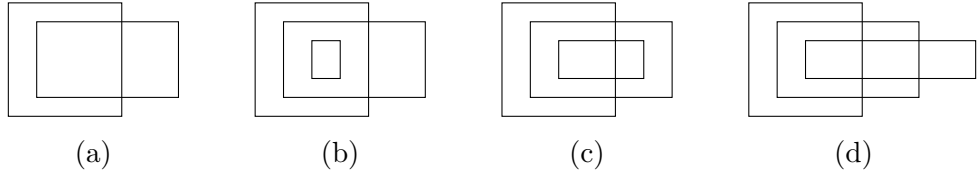

The maximum independent set problem is to find, given an input graph $G$, an independent set in $G$ of maximum size (or maximum weight, when the vertices of $G$ are equipped with weights).
Although the problem is \classNP-hard in general, it might become tractable when the input graphs are ``well structured'', i.e., restricted to a graph class with some strong structural properties.
Indeed, we know many examples of graph classes in which the maximum independent set problem can be solved in polynomial time.
It turns out that these classes are also $\chi$-bounded.
A prominent example of a hereditary graph class with both discussed properties is the class of perfect graphs~\cite{GLS84}.
Numerous other examples can be found among classes that are defined by geometric representations \cite{Gav73,Gya85}, by bounded width parameters \cite{ABC+25,CMR00,DK12,Yol18}, or by excluded induced subgraphs \cite{ACPR24,GKPP22,Gya87,LVV14,LM08,Min80,Sbi80}.

Motivated by such examples, Thomassé, Trotignon, and Vušković~\cite{TTV17} asked whether \emph{every} hereditary class of graphs that admits a polynomial-time algorithm for the maximum independent set problem is $\chi$-bounded.
Being able to construct a strict frame representation, we can solve the maximum independent set problem (with weights) in Burling graphs in polynomial time via a known dynamic programming scheme due to Gavril \cite{Gav73,Gav00}, which was first used in the context of so-called circle graphs.
Therefore, since the class of Burling graphs is hereditary and not $\chi$-bounded, we answer the above-mentioned question in~\cite{TTV17} in the negative.

\begin{theorem}
\label{thm:mis-burling}
There is a polynomial-time algorithm for the maximum independent set problem in Burling graphs with weights.
\end{theorem}

The rest of this paper is organized as follows.
In \Cref{sec:prelims}, we formally define strict frame representations and Burling sets---relational structures, first defined in~\cite{PT23}, which provide a convenient abstraction over strict frame representations, and we also state and prove some basic properties of Burling sets.
In \Cref{sec:recognition}, we present a polynomial-time algorithm that decides whether a given graph is a Burling graph and if it is, constructs a representation by a Burling set.
In \Cref{sec:construction}, we describe a polynomial-time algorithm that turns a Burling set into a strict frame representation.
In \Cref{sec:mis}, we present a polynomial-time algorithm to find a maximum independent set given a Burling graph (with weights) and its representation by a Burling set.
We conclude in \Cref{sec:conclusion} with various comments on related research and open problems.

\section{Frames and Burling sets}
\label{sec:prelims}

The definition of the Burling sequence $B_1,B_2,\ldots$ in modern terms can be found in \cite{CELO16,PT23}.
We will not need it in this paper.
Instead, we will use equivalent characterizations of Burling graphs due to Pournajafi and Trotignon~\cite{PT23}.
We start with necessary definitions.

A \emph{frame} is the boundary of an axis-parallel rectangle in the plane, that is, a set of the form
\[F(\ell,r,b,t)=(\set{\ell,r}\times[b,t])\cup([\ell,r]\times\set{b,t}){,}\]
where $\ell,r,b,t\in\setR$, $\ell<r$, and $b<t$.
A \emph{frame representation} of a graph $G$ is a family of frames $\set{F_v}_{v\in V(G)}$, where $F_v=F(\ell_v,r_v,b_v,t_v)$ for all vertices $v$ of $G$, such that the following conditions are satisfied for every pair of distinct vertices $u$ and $v$ of $G$:
\begin{itemize}
\item none of the corners $(\ell_u,b_u)$, $(\ell_u,t_u)$, $(r_u,b_u)$, $(r_u,t_u)$ of $F_u$ lies on $F_v$;
\item $uv$ is an edge of $G$ if and only if the frames $F_u$ and $F_v$ intersect.
\end{itemize}

Rephrasing \cite[Definition~6.2]{PT23}, we call an indexed family of frames $\set{F_i}_{i\in I}$, where $F_i=F(\ell_i,r_i,b_i,t_i)$ for all $i\in I$, \emph{strict} if it satisfies the following conditions:
\begin{itemize}
\item for every pair of distinct indices $i,j\in I$, if the frames $F_i$ and $F_j$ intersect, then
\[\ell_i<\ell_j<r_i<r_j\enspace\text{and}\enspace b_i<b_j<t_j<t_i\quad\text{or}\quad\ell_j<\ell_i<r_j<r_i\enspace\text{and}\enspace b_j<b_i<t_i<t_j{;}\]
\item there is no triple of indices $i,j,k\in I$ such that
\[\ell_i<\ell_j<\ell_k<r_i<r_j\enspace\text{and}\enspace b_i<b_j<b_k<t_k<t_j<t_i{.}\]
\end{itemize}
The former means that every intersecting pair of frames in the family looks like in \Cref{fig:strict}~(a), and the latter means that the family avoids the configurations illustrated in \Cref{fig:strict}~(b)--(d).

A graph is a \emph{strict frame graph} if it admits a strict frame representation.%
\footnote{We note that allowing the configuration in \Cref{fig:strict}~(b) in the definition of strict frame representation would lead to the same class of graphs, because the vertices represented inside two intersecting frames must be disconnected from the rest of the graph.}
Observe that the class of strict frame graphs is hereditary (that is, an induced subgraph of a strict frame graph is a strict frame graph) and triangle-free (that is, no strict frame graph contains three pairwise adjacent vertices).

A \emph{relation} on a set $S$ is a set ${\rel}\subseteq S\times S$.
We write $x\rel y$ to denote that $(x,y)\in{\rel}$, and we write ${\rel}|_U$ for the relation ${\rel}\cap(U\times U)$ on a subset $U$ of $S$.
A \emph{cycle} in a relation $\rel$ on $S$ is a non-empty set $C\subseteq S$ whose elements can be ordered as $x_1,\ldots,x_{\size{C}}$ so that $x_1\rel\cdots\rel x_{\size{C}}\rel x_1$ (in particular, $x_1\rel x_1$ when $\size{C}=1$).
A relation $\rel$ on $S$ is
\begin{itemize}
\item\emph{irreflexive} if there is no $x\in S$ with $x\rel x$;
\item\emph{acyclic} if there is no cycle in $\rel$ (which implies irreflexivity);
\item\emph{transitive} if $x\rel y$ and $y\rel z$ imply $x\rel z$ for all $x,y,z\in S$.
\end{itemize}
A \emph{strict partial order} is a relation that is irreflexive and transitive (which implies acyclicity).

Following \cite[Definition~5.1]{PT23}, we call a triple $(S,\prec,\adj)$ a \emph{Burling set} if $S$ is a non-empty finite set, $\prec$ is a strict partial order on $S$, $\adj$ is an acyclic relation on $S$, and the following conditions, called \emph{axioms}, hold for all $x,y,z\in S$:
\begin{axioms}
\item\label{axiom:1} if $x\prec y$, and $x\prec z$, and $y\neq z$, then $y\prec z$ or $z\prec y$;
\item\label{axiom:2} if $x\adj y$, and $x\adj z$, and $y\neq z$, then $y\prec z$ or $z\prec y$;
\item\label{axiom:3} if $x\adj y$ and $x\prec z$, then $y\prec z$;
\item\label{axiom:4} if $x\adj y$ and $y\prec z$, then $x\adj z$ or $x\prec z$.
\end{axioms}
\Cref{axiom:3} and irreflexivity of $\prec$ imply that the relations $\prec$ and $\adj$ are disjoint.
Sometimes it is convenient to consider transitivity of $\prec$ as an additional axiom:
\begin{axioms}[resume]
\item\label{axiom:5} if $x\prec y$ and $y\prec z$, then $x\prec z$.
\end{axioms}

Every acyclic relation $\rel$ on a non-empty finite set $S$ gives rise to an undirected graph $(S,\edgeset{\rel})$ with vertex set $S$ and edge set $\edgeset{\rel}=\{xy\colon x,y\in S$ and $x\rel y\}$.

\begin{theorem}[{\cite[Theorems 5.7 and~6.6]{PT23}}]
\label{thm:characterization}
For every graph\/ $G$, the following are equivalent:
\begin{enumerate}
\item $G$ is a Burling graph;
\item $G$ is a strict frame graph;
\item $G=(S,\edgeset{\adj})$ for some Burling set\/ $(S,\prec,\adj)$.
\end{enumerate}
\end{theorem}

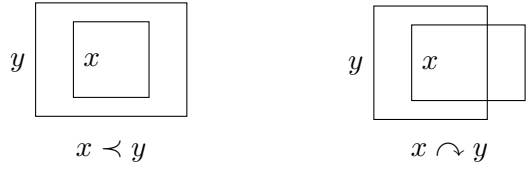
\begin{figure}[t]
\begin{center}
\begin{tikzpicture}[scale=0.5]
  \draw (0,0) rectangle (4,3);
  \draw (1,0.5) rectangle (3,2.5);
  \node[right] at (1,1.5) {$x$};
  \node[left] at (0,1.5) {$\smash[b]{y}$};
  \node[below,yshift=-5pt] at (2,0) {$x\prec y$};
\end{tikzpicture}\hskip 2cm
\begin{tikzpicture}[scale=0.5]
  \draw (0,0) rectangle (3,3);
  \draw (1,0.5) rectangle (4,2.5);
  \node[right] at (1,1.5) {$x$};
  \node[left] at (0,1.5) {$\smash[b]{y}$};
  \node[below,yshift=-5pt] at (2,0) {$x\adj y$};
\end{tikzpicture}
\end{center}
\caption{Correspondence between the relations $\prec$ and $\adj$ in a Burling set and configurations of pairs of frames in its strict frame representation.}
\label{fig:correspondence}
\end{figure}

A strict family of frames $\set{F_x}_{x\in S}$, where $F_x=F(\ell_x,r_x,b_x,t_x)$ for $x\in S$, is a \emph{strict frame representation} of a Burling set $(S,\prec,\adj)$ if the following conditions hold for every pair of distinct elements $x,y\in S$ (see \Cref{fig:correspondence}):
\begin{alignat*}{2}
x&\prec y&&\quad\text{if and only if}\quad\ell_y<\ell_x<r_x<r_y\enspace\text{and}\enspace b_y<b_x<t_x<t_y{;}\\
x&\adj y&&\quad\text{if and only if}\quad\ell_y<\ell_x<r_y<r_x\enspace\text{and}\enspace b_y<b_x<t_x<t_y{.}
\end{alignat*}
Clearly, every strict frame representation of a Burling set $(S,\prec,\adj)$ is a strict frame representation of the graph $(S,\edgeset{\adj})$, for which only the second condition above is required.

It is straightforward to verify that the conditions above turn every strict family of frames $\set{F_x}_{x\in S}$ into a Burling set $(S,\prec,\adj)$; see \cite[Lemma~6.4]{PT23}.
Conversely, it is not difficult to construct a strict frame representation for any given Burling set---we provide an algorithm for this task in \Cref{sec:construction}.
Burling sets with their \crefrange{axiom:1}{axiom:4} thus provide an abstraction of key structural properties of strict families of frames.
In particular, the core difficulty in constructing a strict frame representation for a given graph lies in constructing the corresponding Burling set---we provide an algorithm for this task in the next section.

We complete this section with auxiliary definitions and lemmas on Burling sets.

\begin{lemma}[{\cite[Lemma~5.3]{PT23}}]
\label{lem:acyclic}
In a Burling set\/ $(S,\prec,\adj)$, the relation\/ ${\prec}\cup{\adj}$ is acyclic.
\end{lemma}

\begin{proof}
Let ${\rel}={\prec}\cup{\adj}$.
Suppose $\rel$ is not acyclic.
Let $C$ be a smallest cycle in $\rel$.
Since $\prec$ and $\adj$ are acyclic, the cycle $C$ contains a pair of elements related in $\prec$ and a pair of elements related in $\adj$.
Consequently, there are $x,y,z\in C$ with $x\adj y\prec z$, which implies $x\rel z$ by \cref{axiom:4}.
We conclude that $C-\set{y}$ is a smaller cycle in $\rel$, contradicting minimality of $C$.
\end{proof}

\begin{lemma}
\label{lem:path}
Let\/ $(S,\prec,\adj)$ be a Burling set and\/ $y\in S$.
Then, for every path\/ $x_0\cdots x_k$ in the graph\/ $(S,\edgeset{\adj})$ such that\/ $x_0\prec y$ and\/ $x_k\nprec y$, there is an index\/ $i\in\set{1,\ldots,k}$ such that\/ $x_i\adj y$.
\end{lemma}

\begin{proof}
Let $i\in\set{1,\ldots,k}$ be minimum such that $x_i\nprec y$.
Thus $x_{i-1}\prec y$.
If $x_{i-1}\adj x_i$, then $x_i\prec y$ by \cref{axiom:3}, which is a contradiction.
Thus $x_i\adj x_{i-1}$.
Now, \cref{axiom:4} yields $x_i\adj y$ or $x_i\prec y$, but the latter is again a contradiction.
Thus $x_i\adj y$.
\end{proof}

We define three distinguished types of elements in a Burling set $(S,\prec,\adj)$; see \Cref{fig:elements}.
An element $s\in S$ is a \emph{root} if there is no $x\in S$ with $s\prec x$ or $s\adj x$.
Since the relation ${\prec}\cup{\adj}$ is acyclic, every Burling set has at least one root.
Intuitively, thinking of a strict frame representation of a Burling set, a frame $F(\ell_s,r_s,b_s,t_s)$ with minimum value of $\ell_s$ represents a root.
A \emph{probe} is an element $p\in S$ such that there is no $x\in S$ with $p\prec x$, or $x\prec p$, or $x\adj p$.
A Burling set may have no probes.
Intuitively, in a strict frame representation, the frame $F(\ell_p,r_p,b_p,t_p)$ representing a probe can be extended arbitrarily far to the right, i.e., $r_p$ can be increased arbitrarily without changing the intersection graph.
An element $q\in S$ is \emph{exposed} if there is no $x\in S$ with $q\prec x$.
In particular, roots and probes are exposed.
Intuitively, in a strict frame representation, a frame represents an exposed element if it touches the outer (infinite) region determined by the geometric drawing of the frames.
If $q$ is exposed, then the Burling set $(S,\prec,\adj)$ can be extended by adding a new element $p$ with $p\adj q$ and with no other relations involving $p$.

\begin{figure}[t]
\begin{center}
\begin{tikzpicture}[xscale=0.75,yscale=0.4]
  \draw (0,0) rectangle (3,7);
  \draw (2.5,0.5) rectangle (5,1.5);
  \draw (2.5,2) rectangle (5,4);
  \draw (3.75,2.5) rectangle (4.5,3.5);
  \draw (2.5,4.5) rectangle (5,6.5);
  \draw (4.5,5) rectangle (6,6);
  \node[left] at (0,3.5) {$a$};
  \node[left] at (2.5,1) {$b$};
  \node[left] at (2.5,3) {$c$};
  \node[left] at (2.5,5.5) {$d$};
  \node[left] at (3.75,3) {$e$};
  \node[left] at (4.5,5.5) {$f$};
\end{tikzpicture}
\end{center}
\caption{Distinguished elements of a Burling set in a frame representation: $a$ is a root, $b$ and $f$ are probes, and all three are exposed; $c$ and $d$ are exposed but neither roots nor probes; $e$ is not exposed.}
\label{fig:elements}
\end{figure}
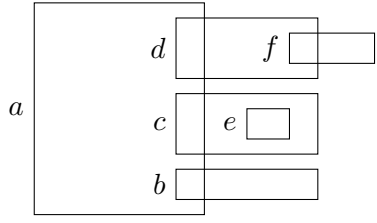

\begin{lemma}
\label{lem:root}
A Burling set\/ $(S,\prec,\adj)$ for which the graph\/ $(S,\edgeset{\adj})$ is connected has a unique root.
\end{lemma}

\begin{proof}
Suppose that, on the contrary, there are multiple roots.
Let $x_0\cdots x_k$ be a shortest path in $(S,\edgeset{\adj})$ between two distinct roots $x_0$ and $x_k$.
Thus $x_1\adj x_0$ and $x_{k-1}\adj x_k$.
Consequently, there is $j\in\set{1,\ldots,k-1}$ such that $x_j\adj x_{j-1}$ and $x_j\adj x_{j+1}$.
This implies that $x_{j-1}\prec x_{j+1}$ or $x_{j+1}\prec x_{j-1}$ by \cref{axiom:2}.
Suppose $x_{j-1}\prec x_{j+1}$.
Since $x_0\nprec x_{j+1}$ (as $x_0$ is a root), \Cref{lem:path} implies that there is $i\in\set{0,\ldots,j-2}$ with $x_i\adj x_{j+1}$.
This shows that $x_0\cdots x_ix_{j+1}\cdots x_k$ is a path in $(S,\edgeset{\adj})$ between $x_0$ and $x_k$ shorter than $x_0\cdots x_k$, which contradicts the choice of the latter as a shortest path.
An analogous contradiction is reached when $x_{j+1}\prec x_{j-1}$.
\end{proof}

\begin{lemma}
\label{lem:outer-join}
Let\/ $(S_1,\prec_1,\adj_1)$ and\/ $(S_2,\prec_2,\adj_2)$ be Burling sets with\/ $S_1\cap S_2=\set{q}$, where\/ $q$ is a root of\/ $(S_1,\prec_1,\adj_1)$ and is exposed in\/ $(S_2,\prec_2,\adj_2)$.
Then\/ $(S_1\cup S_2,\:{\prec_1}\cup{\prec_2},\:{\adj_1}\cup{\adj_2})$ is a Burling set in which every root of\/ $(S_2,\prec_2,\adj_2)$ remains a root, every probe of\/ $(S_1,\prec_1,\adj_1)$ or\/ $(S_2,\prec_2,\adj_2)$ other than\/ $q$ remains a probe, and\/ $q$ remains exposed.
\end{lemma}

\begin{proof}
Let $S=S_1\cup S_2$, ${\prec}={\prec_1}\cup{\prec_2}$, and ${\adj}={\adj_1}\cup{\adj_2}$.
Acyclicity of $\adj$ follows from acyclicity of $\adj_1$ and $\adj_2$ by the assumption that $\size{S_1\cap S_2}=1$.
Irreflexivity of $\prec$ is clear.
To prove that $(S,\prec,\adj)$ is a Burling set, it remains to verify \crefrange{axiom:1}{axiom:5}.
Each of them involves a triple of elements $x,y,z\in S$ and follows from the respective condition on $(S_1,\prec_1,\adj_1)$ or $(S_2,\prec_2,\adj_2)$ when $\set{x,y,z}\subseteq S_1$ or $\set{x,y,z}\subseteq S_2$.
Thus, consider a triple $x,y,z\in S$ with $\set{x,y,z}\nsubseteq S_1$ and $\set{x,y,z}\nsubseteq S_2$.
Let ${\rel}={\prec}\cup{\adj}$.
If $x\rel y$ and $x\rel z$, then $x=q$, and the relation of $q$ with the element $y$ or $z$ in $S_1$ contradicts the assumption that $q$ is a root of $(S_1,\prec_1,\adj_1)$.
Hence, \crefrange{axiom:1}{axiom:3} are satisfied vacuously for $x,y,z$.
If $x\rel y$ and $y\prec z$, then $y=q$, and the relation $q\prec z$ contradicts the assumption that $q$ is exposed in $(S_1,\prec_1,\adj_1)$ and in $(S_2,\prec_2,\adj_2)$.
Hence, \cref{axiom:4,axiom:5} are satisfied vacuously for $x,y,z$.
This shows that $(S,\prec,\adj)$ is a Burling set.
The statements on roots, probes, and exposure of $q$ being preserved follows directly from the definitions and the fact that ${\prec}={\prec_1}\cup{\prec_2}$ and ${\adj}={\adj_1}\cup{\adj_2}$.
\end{proof}

\begin{lemma}
\label{lem:inner-join}
Let\/ $(S_1,\prec_1,\adj_1)$ and\/ $(S_2,\prec_2,\adj_2)$ be Burling sets with the following properties:
\begin{itemize}
\item $S_1\cap S_2$ is a non-empty set of probes in both\/ $(S_1,\prec_1,\adj_1)$ and\/ $(S_2,\prec_2,\adj_2)$;
\item there is a set\/ $S_2'\subseteq S_2$ such that\/ $S_2'=\set{x\in S_2\colon q\adj_2x}$ for all\/ $q\in S_1\cap S_2$.
\end{itemize}
Then\/ $(S_1\cup S_2,\:{\prec_1}\cup{\prec_2}\cup((S_1-S_2)\times S_2'),\:{\adj_1}\cup{\adj_2})$ is a Burling set in which every root of\/ $(S_2,\prec_2,\adj_2)$ that is not in\/ $S_1$ remains a root and every probe of\/ $(S_2,\prec_2,\adj_2)$ remains a probe.
\end{lemma}

\begin{proof}
Let $S=S_1\cup S_2$, $Q=S_1\cap S_2$, ${\prec}={\prec_1}\cup{\prec_2}\cup((S_1-S_2)\times S_2')$, and ${\adj}={\adj_1}\cup{\adj_2}$.
Let ${\rel}={\prec}\cup{\adj}$.
Since $Q$ is a common set of probes for both Burling sets, there are no $q\in Q$ and $x\in S$ with $q\prec x$ or $x\rel q$.
This implies acyclicity of $\adj$.
Irreflexivity of $\prec$ is clear.
To prove that $(S,\prec,\adj)$ is a Burling set, it remains to verify \crefrange{axiom:1}{axiom:5}.
Each of them involves a triple of elements $x,y,z\in S$ and follows from the respective condition on $(S_1,\prec_1,\adj_1)$ or $(S_2,\prec_2,\adj_2)$ when $\set{x,y,z}\subseteq S_1$ or $\set{x,y,z}\subseteq S_2$.
Thus, consider a triple $x,y,z\in S$ with $\set{x,y,z}\nsubseteq S_1$, $\set{x,y,z}\nsubseteq S_2$, $x\rel y$, and $x\rel z$ or $y\rel z$.
The latter conditions imply that $y,z\notin Q$.
We proceed by case distinction, in each case proving that every axiom, which has the form of an implication, holds in its conclusion or fails in its premise for the triple $x,y,z$.

If $x\in S_2-Q$, then the assumptions that $x\rel y$ and $x\rel z$ or $y\rel z$ imply $y,z\in S_2-Q$, contradicting the assumption that $\set{x,y,z}\nsubseteq S_2$.
Thus $x\in S_1$.
Now, if $x\in Q$, $y\in S_1-Q$, and $z\in S_2-Q$, then the assumption that $x\rel y$ implies $x\adj y$, and the assumption that $x\rel z$ or $y\rel z$ implies $x\adj z\in S_2'$ and consequently $y\prec z$, so the premises of \cref{axiom:1,axiom:3,axiom:4,axiom:5} fail and the conclusion of \cref{axiom:2} holds for $x,y,z$.
If $x\in Q$, $y\in S_2-Q$, and $z\in S_1-Q$, then the assumption that $x\rel y$ implies $x\adj y\in S_2'$ and consequently $z\prec y$, and the assumption that $x\rel z$ or $y\rel z$ implies $x\adj z$, with the same consequence on the axioms as above.
If $x,y\in S_1-Q$ and $z\in S_2-Q$, then the assumption that $x\rel z$ or $y\rel z$ implies $z\in S_2'$, so that $x\prec z$ and $y\prec z$, making the conclusion of every axiom hold for $x,y,z$.
Finally, suppose $x\in S_1-Q$ and $y\in S_2-Q$.
The assumption that $x\rel y$ implies $y\in S_2'$, so that $x\prec y$.
This makes the premises of \crefrange{axiom:2}{axiom:4} fail for $x,y,z$.
If additionally $z\in S_1-Q$, then $z\prec y$, making the conclusion of \cref{axiom:1} hold and the premise of \cref{axiom:5} fail for $x,y,z$.
Now, suppose $z\in S_2-Q$.
Let $q\in Q$.
It follows that $q\adj y$.
If $z\in S_2'$, then $x\prec z$, making the conclusion of \cref{axiom:5} hold, and $q\adj z$, implying $y\prec z$ or $z\prec y$, and making the conclusion of \cref{axiom:1} hold for $x,y,z$.
If $z\notin S_2'$, then $x\rel z$ is impossible, so the premise of \cref{axiom:1} fails, and so does the premise of \cref{axiom:5}, because $y\prec z$ would imply $q\adj z$ or $q\prec z$ by \cref{axiom:4}, either of which is a contradiction.

We have shown that $(S,\prec,\adj)$ is a Burling set.
The statement on roots and probes being preserved follows directly from the definitions and the construction of $\prec$ and $\adj$.
\end{proof}

\section{Recognition}
\label{sec:recognition}

For this entire section, we fix a graph $G$ for which we want to decide whether it is a Burling graph and if it is, to construct a Burling set $(V,\prec,\adj)$ such that $G=(V,\edgeset{\adj})$.
All graph-theoretic terms will refer to the fixed graph $G$.
In particular, a \emph{vertex} or an \emph{edge} will always mean a vertex or an edge of $G$.
We will be assuming that $G$ is triangle-free, because a graph containing a triangle is never a Burling graph, and whether a graph contains a triangle can be decided trivially in polynomial time.

We will use the following notation and terminology.
Let $V$ be the set of vertices (of $G$).
A~set $S\subseteq V$ is \emph{connected} if $S\neq\emptyset$ and the induced subgraph $G[S]$ is connected.
A \emph{component} of a set $S\subseteq V$ is an inclusion-maximal connected subset of $S$.
A \emph{neighbor} of a vertex $v$ is a vertex with an edge to $v$.
For a vertex $v$, let $N(v)$ denote the set of neighbors of $v$, and let $N[v]=N(v)\cup\set{v}$.
A \emph{neighbor} of a set $S\subseteq V$ is a vertex in $V-S$ with an edge to at least one vertex in $S$.
For a set $S\subseteq V$, let $N(S)$ denote the set of neighbors of $S$, and let $N[S]=S\cup N(S)$.

A set $S'\subseteq V$ is \emph{homogeneous} for a set $S\subseteq V$ if $S'\subseteq N(x)$ or $S'\cap N(x)=\emptyset$ for every $x\in S$.
In other words, $S'$ is homogeneous for $S$ if all vertices in $S'$ have the same neighbors in $S$.
A~family $\famC$ of pairwise disjoint subsets of $V$ is \emph{nested} if every pair of distinct members $C_1,C_2\in\famC$ satisfies at least one of the following conditions:
\begin{itemize}
\item $N(C_1)\subseteq N(C_2)$ and the set $N(C_1)$ is homogeneous for $C_2$;
\item $N(C_2)\subseteq N(C_1)$ and the set $N(C_2)$ is homogeneous for $C_1$;
\item $N(C_1)\cap N(C_2)=\emptyset$.
\end{itemize}
A \emph{nesting order} on a nested family $\famC$ is a strict partial order $<$ on $\famC$ with the following properties for all distinct members $C_1,C_2\in\famC$:
\begin{itemize}
\item if $C_1<C_2$, then $N(C_1)\subseteq N(C_2)$ and the set $N(C_1)$ is homogeneous for $C_2$;
\item if neither $C_1<C_2$ nor $C_2<C_1$, then $N(C_1)\cap N(C_2)=\emptyset$.
\end{itemize}
In the latter case, $N(C_1)$ is trivially homogeneous for $C_2$, and so is $N(C_2)$ for $C_1$.
It is clear that every nested family has a nesting order.
The definition directly implies a polynomial-time algorithm to test whether a family $\famC$ is nested and if it is, to construct a nesting order on $\famC$.

A \emph{Burling structure} is a Burling set $(U,\prec,\adj)$ such that $U\subseteq V$ and $(U,\edgeset{\adj})$ is an induced subgraph of $G$, that is, for all $x,y\in U$, there is an edge $xy$ in $G$ if and only if $x\adj y$ or $y\adj x$.
Such a Burling structure is \emph{around} a connected set $S$ if $U=N[S]$.

We are ready to provide an algorithm that decides whether $G$ is a Burling graph and if it is, produces a Burling set $(V,\prec,\adj)$ such that $G=(V,\edgeset{\adj})$.
The algorithm applies the dynamic programming technique, solving the following \emph{subproblems} on the way:
\begin{itemize}
\item an \emph{unrooted subproblem} $\Unrooted(X,S)$ asks for a Burling structure around $S$ with $N(S)$ a set of probes, where $X=\emptyset$ or $X=N[v]$ for some vertex $v$, and $S$ is a component of $V-X$;
\item a \emph{rooted subproblem} $\Rooted(X,r,S)$ asks for a Burling structure around $S$ with $r$ a root and $N(S)-\set{r}$ a set of probes, where $X=\emptyset$ or $X=N[v]$ for some vertex $v$, $r\in V-X$, and $S$ is a component of $V-(X\cup\set{r})$ containing at least one neighbor of $r$.
\end{itemize}
The \emph{size} of such a subproblem is the size of $S$.
A \emph{solution} to a subproblem is a Burling structure being asked for; if it does not exist, then the subproblem has no solution.
It is clear from the description above that the number of subproblems is polynomial in the number of vertices.
Careful analysis shows that if $G$ has $n$ vertices, then there are only $\Oh(n^2)$ subproblems---we omit the details.

Below, we describe algorithms to solve a subproblem $\Unrooted(X,S)$ or $\Rooted(X,r,S)$ assuming that all smaller subproblems have been already solved.
In these descriptions, \emph{failing} means that the algorithm terminates computations and reports no solution to the subproblem.

We start by explaining the intuition behind the algorithm for an unrooted problem $\Unrooted(X,S)$, describing it in terms of a strict frame representation.
Recall that $X=N[v]$ for some vertex $v$ or $X=\emptyset$.
We focus on the former setting.
Suppose that we have already fixed the frame $F_v$ for $v$ and decided that the component $S$ of $G-X$ should be represented inside $F_v$.
Recall that in $\Unrooted(X,S)$ we ask not only for a representation of $G[S]$, but of $G[N[S]]$, requiring additionally that the vertices in $N(S)$ are probes.
The frames representing such probes can be easily extended to the right in order to intersect the frame of $v$, which provides a clean intersection between the representation of $N[S]$ we are building in the current call and the representation of the rest of the graph that has been built so far.
Our goal is to find a vertex $r$ in $S$ that can be the root of $G[N[S]]$.
An obvious condition that the root must satisfy is as follows: for each component $C$ of $S-\set{r}$, the graph $G[C]$ has a representation with $r$ a root and $N(C)-\set{r}$ a set of probes; this can be verified by solving the rooted subproblem $\Rooted(X,r,C)$.
Another, more subtle condition is that no two frames representing vertices from different components of $S-\set{r}$ are nested, which implies that no probe in $N(S)$ can have neighbors in two distinct components of $S-\set{r}$; consult \Cref{fig:unrooted}.
If these conditions are satisfied, it is not difficult to assemble representations of $G[N[C]]$ for all components $C$ of $S-\set{r}$ into a desired representation of $G[N[S]]$.

Now, we turn this intuition into a formal description of the algorithm.

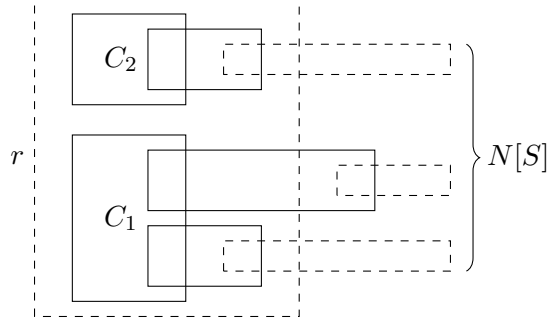
\begin{figure}[t]
\begin{center}
\begin{tikzpicture}[xscale=0.5,yscale=0.4]
  \draw[dashed] (0,0) rectangle (7,10.5);
  \draw (1,0.5) rectangle (4,6);
  \draw (3,1) rectangle (6,3);
  \draw (3,3.5) rectangle (9,5.5);
  \draw (1,7) rectangle (4,10);
  \draw (3,7.5) rectangle (6,9.5);
  \draw[dashed] (5,1.5) rectangle (11,2.5);
  \draw[dashed] (8,4) rectangle (11,5);
  \draw[dashed] (5,8) rectangle (11,9);
  \node[left] at (0,5.25) {$r$};
  \node[left] at (3,3.25) {$C_1$};
  \node[left] at (3,8.5) {$C_2$};
  \draw[decorate,decoration={brace,amplitude=5pt,mirror,raise=6pt}] (11,1.5) -- (11,9) node[midway,right,xshift=10pt]{$N[S]$};
\end{tikzpicture}
\end{center}
\caption{The second condition in the algorithm for $\Unrooted(X,S)$ with $C_1$ and $C_2$ the components of $S-\set{r}$.}
\label{fig:unrooted}
\end{figure}

\begin{algorithm}{$\Unrooted(X,S)$}
\item\label{step:unrooted-1}
Find a vertex $r\in S$ that satisfies the following two conditions, or fail if there is no such vertex:
\begin{itemize}
\item for every component $C$ of $S-\set{r}$, the subproblem $\Rooted(X,r,C)$ has a solution;
\item for every $p\in N(S)$, the set $N(p)\cap N[S]$ is contained in $C\cup\set{r}$ for some component $C$ of $S-\set{r}$ or is equal to $\set{r}$.
\end{itemize}
\item\label{step:unrooted-2}
Let $(N[C],\prec_C,\adj_C)$ be a solution to $\Rooted(X,r,C)$ for each component $C$ of $S-\set{r}$.
Let $\prec$ and $\adj$ be the unions of the relations $\prec_C$ and $\adj_C$ (respectively) over all components $C$ of $S-\set{r}$ with an additional related pair $p\adj r$ for every $p\in N(S)$ such that $N(p)\cap N[S]=\set{r}$.
Report $(N[S],\prec,\adj)$ as a solution to $\Unrooted(X,S)$.
\end{algorithm}

\begin{lemma}
\label{lem:unrooted}
The algorithm to solve\/ $\Unrooted(X,S)$ is correct assuming that the smaller subproblems have been correctly solved.
\end{lemma}

\begin{proof}
First, observe that $N(S)$ is an independent set.
Indeed, if $X=\emptyset$, then $N(S)=\emptyset$, and if $X=N[v]$ for a vertex $v$, then $N(S)\subseteq N(v)$, and $N(v)$ is an independent set by the assumption that $G$ is triangle-free.

Suppose that the algorithm reports $(N[S],\prec,\adj)$ as a solution to $\Unrooted(X,S)$.
Let $r$ be the vertex found by the algorithm in \cref{step:unrooted-1}.
Let a component of $S-\set{r}$ be called simply a \emph{component}.
Let $P_r=\set{p\in N(S)\colon N(p)\cap S=\set{r}}$.
By construction, the following Burling structures are the restrictions of $(N[S],\prec,\adj)$ to the sets $P_r\cup\set{r}$ and $N[C]$ for all components $C$, respectively:
\begin{itemize}
\item $(P_r\cup\set{r},\:\emptyset,\:P_r\times\set{r})$ with $r$ a root and $P_r$ a set of probes;
\item $(N[C],\prec_C,\adj_C)$ with $r$ a root and $N(C)-\set{r}$ a set of probes---a solution to $\Rooted(X,r,C)$ chosen by the algorithm in \cref{step:unrooted-2}, for each component $C$.
\end{itemize}
Their union gives rise to $(N[S],\prec,\adj)$.
The sets $P_r\cup\set{r}$ and $N[C]$ over all components $C$ are pairwise disjoint except at the common root $r$.
Therefore, we can repeatedly invoke \Cref{lem:outer-join} to infer that $(N[S],\prec,\adj)$ is a Burling set in which $r$ is a root and the union of the sets $P_r$ and $N[C]-\set{r}$ over all components $C$ is a set of probes.
The latter union is equal to $N(S)$.
Moreover, since $N(S)$ is an independent set, every edge between two vertices from $N[S]$ lies in $N[C]$ for some component $C$ or connects $r$ with a vertex in $P_r$.
Thus $(N[S],\edgeset{\adj})$ is an induced subgraph of $G$.
We conclude that $(N[S],\prec,\adj)$ is a correct solution to $\Unrooted(X,S)$.

Now, suppose there is a solution $(N[S],\prec,\adj)$ to $\Unrooted(X,S)$.
Let $r$ be a root of $(N[S],\prec,\adj)$.
Let a component of $S-\set{r}$ be called simply a \emph{component}.
For every component $C$, the triple $(N[C],{\prec}|_{N[C]},{\adj}|_{N[C]})$ is a Burling structure around $C$ with $r$ a root and $N(C)-\set{r}$ a set of probes---a solution to $\Rooted(X,r,C)$.
Let $p\in N(S)$, and suppose towards a contradiction that $p$ is a neighbor of two distinct components $C_1$ and $C_2$.
Since $p$ is a probe in $(N[S],\prec,\adj)$, there are $x_1\in C_1$ and $x_2\in C_2$ such that $p\adj x_1$ and $p\adj x_2$.
\Cref{axiom:2} entails $x_1\prec x_2$ or $x_2\prec x_1$; assume the former, without loss of generality.
By \Cref{lem:path}, since $r\nprec x_2$, a path from $x_1$ to $r$ in $C_1\cup\set{r}$ must contain a vertex $x$ such that $x\adj x_2$.
This is a contradiction, because neither $r\adj x_2$ nor any vertex in $C_1$ is a neighbor of $x_2$.
This shows that the sets $N(S)\cap N(C)$ for all components $C$ are pairwise disjoint.
Thus, the vertex $r$ satisfies both conditions verified by the algorithm in \cref{step:unrooted-1}, so the algorithm proceeds to \cref{step:unrooted-2}, where it computes (correctly, as we have shown in the first part of the proof) some solution to $\Unrooted(X,S)$.
\end{proof}

Now, we proceed with the description of the algorithm for a rooted subproblem $\Rooted(X,r,S)$.
Again, we start with some intuition described in terms of strict frame representations; see \Cref{fig:rooted} for an illustration.
We look for a representation of $G[N[S]]$ with $r$ a root and $N(S)-\set{r}$ a set of probes.
Consider any such representation, and let $F_r$ be the frame representing $r$.
Each component $C$ of $S-N(r)$ is represented entirely inside or entirely outside $F_r$.
Call $C$ an \emph{inner component} in the former and an \emph{outer component} in the latter case.
Each inner component $C$ satisfies $N(C)\subseteq N(r)$, because the representation of any path from $C$ to a vertex from $G-N[C]$ must cross $F_r$.
In the representation restricted to $N[C]$, the vertices in $N(S)$ are represented as probes, so that they can intersect $F_r$
Moreover, the way how inner components attach to their neighborhoods is very restricted---they form a nested family.
Each outer component $C$ has exactly one neighbor in $N(r)$, call it $q_C$.
In the representation restricted to $N[C]$, $q_C$ is a (unique) root, and the vertices in $N(C)-\set{q_C}$, which belong to $N(S)-\set{r}$, are probes, as required by the considered representation of $G[N[S]]$.
There is one more condition that must be satisfied by the vertices in $N(S)-\set{r}$, like in the unrooted problem.
These necessary conditions on inner and outer components and their neighborhoods can be verified looking only at the graph $G$, without the representation.
Moreover, representations of the graphs $G[N[C]]$ obtained from solving the subproblems $\Unrooted(N[r],C)$ (for inner components $C$) and $\Rooted(X,q_C,C)$ (for outer components $C$) can be combined into a desired representation of $G[N[S]]$.

\begin{figure}[t]
\begin{center}
\begin{tikzpicture}[scale=0.75]
  \draw (-1,-0.75) rectangle (7,7.25);
  \draw[fill=lightgray,even odd rule] (1,-0.5) rectangle (6.5,7) (1.5,0) rectangle (6,6.5);
  \draw[fill=lightgray,even odd rule] (2.5,0.25) rectangle (5.5,4) (3,0.75) rectangle (5,3.5);
  \draw[fill=lightgray] (9.5,1.25) rectangle (11,1.75);
  \draw[fill=lightgray] (9.5,2.5) rectangle (11,3);
  \draw[fill=lightgray] (9.5,4.75) rectangle (11,5.25);
  \draw[fill=lightgray] (9.5,5.5) rectangle (11,6);
  \draw (4.5,1) rectangle (10,2);
  \draw (4.5,2.25) rectangle (10,3.25);
  \draw (5.5,4.5) rectangle (10,6.25);
  \node[left] at (-1,3.25) {$r$};
  \node[left] at (1,3.25) {$C_1$};
  \node[left] at (2.5,2.125) {$C_2$};
  \node[right] at (11,1.5) {$C_3$};
  \node[right] at (11,2.75) {$C_4$};
  \node[right] at (11,5) {$C_5$};
  \node[right] at (11,5.75) {$C_6$};
  \node[left] at (4.5,1.5) {$q_{C_3}$};
  \node[left] at (4.5,2.75) {$q_{C_4}$};
  \node[left] at (5.5,5.375) {$q_{C_5}=q_{C_6}$};
\end{tikzpicture}
\end{center}
\caption{A solution to a subproblem $\Rooted(X,r,S)$.
The graph induced by $S-N(r)$ has six components $C_1,\ldots,C_6$, of which $C_1$ and $C_2$ are inner while $C_3,\ldots,C_6$ are outer.
Gray areas are where particular components are represented.}
\label{fig:rooted}
\end{figure}
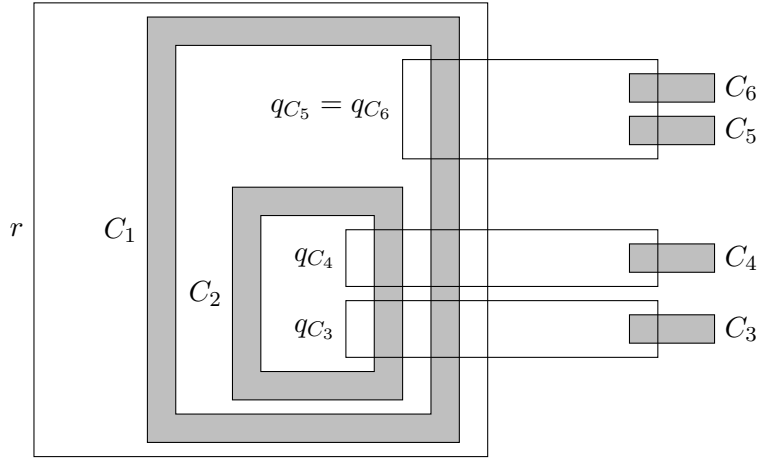

\begin{algorithm}{$\Rooted(X,r,S)$}
\item\label{step:rooted-1}
Classify each component $C$ of $S-N(r)$ as
\begin{itemize}
\item an \emph{inner component} if $N(C)\subseteq N(r)$ and the subproblem $\Unrooted(N[r],C)$ has a solution;
\item an \emph{outer component} if $\size{N(C)\cap N(r)}=1$ and the subproblem $\Rooted(X,q_C,C)$ has a solution, where $q_C$ is the vertex such that $N(C)\cap N(r)=\set{q_C}$.
\end{itemize}
Fail if some component of $S-N(r)$ cannot be classified either way.
For components that can be classified both ways, choose arbitrarily.
\item\label{step:rooted-2}
Check the following condition and fail if it does not hold: for every $p\in N(S)-\set{r}$, the set $N(p)\cap N[S]$ is contained in the union of $\set{r}$ and all inner components, or is contained in $C\cup\set{q_C}$ for some outer component $C$, or is equal to $\set{q}$ for some $q\in N(r)\cap S$.
\item\label{step:rooted-3}
Let $\famC$ be the family of all inner components.
Check whether $\famC$ is nested and fail if not.
Compute a nesting order $<$ on $\famC$.
\item\label{step:rooted-4}
Let $(N[C],\prec_C,\adj_C)$ be a solution to $\Unrooted(N[r],C)$ for every inner component $C$ or to $\Rooted(X,q_C,C)$ for every outer component $C$.
Let $\prec$ and $\adj$ be the unions of the relations $\prec_C$ and $\adj_C$ (respectively) over all components $C$ with the following additional related pairs:
\begin{itemize}
\item $x\prec\alignto{y}{l}{r}$ for all $x\in C$, for all $C\in\famC$;
\item $x\prec y$ for all $x\in C_1$ and all neighbors $y$ of $N(C_1)$ in $C_2$, for all $C_1,C_2\in\famC$ with $C_1<C_2$;
\item $\alignto{p}{r}{q}\adj\alignto{q}{l}{r}$ for all $q\in N(r)\cap N[S]$;
\item $p\adj q$ for all $q\in N(r)\cap S$ and all $p\in N(S)-\set{r}$ such that $N(p)\cap N[S]=\set{q}$.
\end{itemize}
Report $(N[S],\prec,\adj)$ as a solution to $\Rooted(X,r,S)$.
\end{algorithm}

\begin{lemma}
\label{lem:rooted}
The algorithm to solve\/ $\Rooted(X,r,S)$ is correct assuming that the smaller subproblems have been correctly solved.
\end{lemma}

\begin{proof}
Like in the proof of \Cref{lem:unrooted}, observe that $N(S)-\set{r}$ is an independent set.
Indeed, if $X=\emptyset$, then $N(S)-\set{r}=\emptyset$, and if $X=N[v]$ for a vertex $v$, then $N(S)-\set{r}\subseteq N(v)$, and $N(v)$ is an independent set by the assumption that $G$ is triangle-free.

First, we prove that if the algorithm reports a solution to $\Rooted(X,r,S)$, then the solution is correct.
Suppose that the algorithm reports $(N[S],\prec,\adj)$ as a solution to $\Rooted(X,r,S)$.
Let $Q=N(r)\cap N[S]$.
Since $G$ is triangle-free, $Q$ is an independent set.

By construction, the following Burling structures are the restrictions of $(N[S],\prec,\adj)$ to the sets $Q\cup\set{r}$ and $N[C]$ for all inner components $C$, respectively:
\begin{itemize}
\item $(Q\cup\set{r},\:\emptyset,\:Q\times\set{r})$ with $r$ a root and $Q$ a set of probes;
\item $(N[C],\prec_C,\adj_C)$ with $N(C)$ a set of probes---a solution to $\Unrooted(N[r],C)$ chosen by the algorithm in \cref{step:rooted-4}, for each inner component $C$.
\end{itemize}
Let $<$ be the nesting order on the inner components computed in \cref{step:rooted-3}.
Enumerate the inner components as $C_1,\ldots,C_k$ so that if $C_i<C_j$, then $i>j$.
It follows that for every $i\in\set{1,\ldots,k}$, the set $N(C_i)$ is homogeneous for $\set{r}$ and for each of the sets $C_1,\ldots,C_{i-1}$.
Let $S_0=Q\cup\set{r}$ and $S_i=Q\cup\set{r}\cup N[C_1]\cup\cdots\cup N[C_i]$ for $i\in\set{1,\ldots,k}$.
It follows that for every $i\in\set{1,\ldots,k}$, we have $S_i-Q=\set{r}\cup C_1\cup\cdots\cup C_i$, and the set $N(C_i)$ (which is a subset of $Q$) is homogeneous for $S_{i-1}-Q$, whence it follows that there is a set $S_{i-1}'\subseteq S_{i-1}-Q$ such that $S_{i-1}'=\set{x\in S_{i-1}\colon q\adj x}$ for all $q\in N(C_i)$.

We prove, for $i\in\set{0,\ldots,k}$ by induction, that $(S_i,{\prec}|_{S_i},{\adj}|_{S_i})$ is a Burling set with $r$ a root and with $Q$ a set of probes.
This holds for $i=0$, as $(S_0,{\prec}|_{S_0},{\adj}|_{S_0})=(Q\cup\set{r},\:\emptyset,\:Q\times\set{r})$.
Now, let $i\in\set{1,\ldots,k}$, and suppose that $(S_{i-1},{\prec}|_{S_{i-1}},{\adj}|_{S_{i-1}})$ is a Burling set with $r$ a root and with $Q$ a set of probes.
We have
\[S_i=S_{i-1}\cup N[C_i]{,}\quad{\prec}|_{S_i}={\prec}|_{S_{i-1}}\cup{\prec_{C_i}}\cup(C_i\times S_{i-1}'){,}\quad\text{and}\quad{\adj}|_{S_i}={\adj}|_{S_{i-1}}\cup{\adj_{C_i}}{.}\]
Since $N(C_i)=N[C_i]\cap Q=N[C_i]\cap S_{i-1}$ and $(N[C_i],\prec_C,\adj_C)$ is a Burling set with $N(C_i)$ a set of probes, we can invoke \Cref{lem:inner-join} to infer that $(S_i,{\prec}|_{S_i},{\adj}|_{S_i})$ is a Burling set with $r$ a root and with $Q$ a set of probes, as needed for the induction step.

Let $P_q=\set{p\in N(S)-\set{r}\colon N(p)\cap N[S]=\set{q}}$ for every $q\in Q\cap S$.
By construction, the following Burling structures are the restrictions of $(N[S],\prec,\adj)$ to the sets $P_q$ for all $q\in Q\cap S$ and $N[C]$ for all outer components $C$, respectively:
\begin{itemize}
\item $(P_q\cup\set{q},\:\emptyset,\:P_q\times\set{q})$ with $q$ a root and $P_q$ a set of probes, for each $q\in Q\cap S$;
\item $(N[C],{\prec}_C,{\adj}_C)$ with $q_C$ a root and $N(C)-\set{q_C}$ a set of probes---a solution to $\Rooted(X,q_C,C)$ chosen by the algorithm in \cref{step:rooted-4}, for each outer component $C$.
\end{itemize}
The union of the Burling set $(S_k,{\prec}|_{S_k},{\adj}|_{S_k})$ and the Burling structures above gives rise to $(N[S],\prec,\adj)$.
The sets $P_q$ for all $q\in Q\cap S$ and $N[C]$ for all outer components $C$ are disjoint from each other and from $S_k$ except at their roots $q$ and $q_C$, respectively.
Since these roots are exposed in $(S_k,{\prec}|_{S_k},{\adj}|_{S_k})$ (and remain exposed), we can repeatedly invoke \Cref{lem:outer-join} to infer that $(N[S],\prec,\adj)$ is a Burling set in which $r$ is a root and the union of the sets $Q-S$, $P_q$ for all $q\in Q\cap S$, and $N(C)-\set{q_C}$ for all outer components $C$ is a set of probes.
The latter union is equal to $N(S)-\set{r}$.
Moreover, since $N(S)-\set{r}$ is an independent set, every edge between two vertices in $N[S]$ lies in $N[C]$ for some component $C$, or connects $r$ with an element of $Q$, or connects $q$ with an element of $P_q$ for some $q\in Q\cap S$.
Thus $(N[S],\edgeset{\adj})$ is an induced subgraph of $G$.
We conclude that $(N[S],\prec,\adj)$ is a correct solution to $\Rooted(X,r,S)$.

Now, we prove that if there is a solution to $\Rooted(X,r,S)$, then the algorithm reports one (which is then correct, as we have proved above).
Let $(N[S],\prec,\adj)$ be a solution to $\Rooted(X,r,S)$, that is, a Burling structure around $S$ with $r$ a root and $N(S)-\set{r}$ a set of probes.
For clarity, we proceed by proving an enumerated series of claims, which eventually lead us to the desired conclusion.

A component $C$ is \emph{inside} $r$ if $x\prec r$ for all $x\in C$, and it is \emph{outside} $r$ if $x\nprec r$ for all $x\in C$.

\begin{enumerate}
\item\label{item:component}
Every component is either inside or outside $r$.
\end{enumerate}
Suppose that, on the contrary, a component $C$ contains vertices $x$ and $y$ such that $x\prec r$ and $y\nprec r$.
By \Cref{lem:path}, the path between $x$ and $y$ in $C$ contains a vertex $z$ such that $z\adj r$.
This contradicts the assumption that $C\subseteq S-N(r)$.

\begin{enumerate}[resume]
\item
The following holds for every $q\in N(r)\cap N[S]$:
\begin{enumerate}
\item\label{item:q-prec} there is no $x\in N[S]$ with $q\prec x$;
\item\label{item:q-adj} every $x\in N[S]-\set{r}$ with $q\adj x$ belongs to a component inside $r$;
\item\label{item:prec-q} every $x\in N[S]$ with $x\prec q$ belongs to a component outside $r$;
\item\label{item:adj-q} every $x\in N[S]$ with $x\adj q$ belongs to a component outside $r$ or to $N(S)-N[r]$.
\end{enumerate}
\end{enumerate}
Since $q\in N(r)$ and $r$ is a root, we have $q\adj r$.
If $q\prec x\in N[S]$, then \cref{axiom:3} entails $r\prec x$, which contradicts $r$ being a root.
This shows \ref{item:q-prec}.
For the proof of \ref{item:q-adj}, suppose $q\adj x\in N[S]-\set{r}$.
Since $N(S)-\set{r}$ is a set of probes in $(N[S],\prec,\adj)$, we have $x\in S$.
\Cref{axiom:2} and the assumption that $r$ is a root imply $x\prec r$, which implies that $x\notin N(r)$ and therefore, by \ref{item:component}, $x$ belongs to a component inside $r$.
For the proof of \ref{item:prec-q}, suppose $N[S]\ni x\prec q$.
It follows that $x$ is not a root or a probe, and $x\notin N(r)$ by \ref{item:q-prec}, so $x\in S-N(r)$.
If $x\prec r$, then \cref{axiom:1} entails either $q\prec r$, which contradicts \ref{item:q-prec}, or $r\prec q$, which contradicts $r$ being a root.
Thus $x\nprec r$ and therefore, by \ref{item:component}, $x$ belongs to a component outside $r$.
For the proof of \ref{item:adj-q}, suppose $N[S]\ni x\adj q$.
It follows that $x$ is not a root, and \ref{item:q-adj} implies $x\notin N(r)$, so $x\in N[S]-N[r]$.
Thus $x\in S-N(r)$ or $x\in N(S)-N[r]$.
In the former case, if $x\prec r$, then \cref{axiom:3} entails $q\prec r$ contradicting \ref{item:q-prec}, so $x\nprec r$ and therefore, by \ref{item:component}, $x$ belongs to a component outside $r$.

\begin{enumerate}[resume]
\item\label{item:inner}
For every component $C$ inside $r$, we have $N(C)\subseteq N(r)$, and $(N[C],{\prec}|_{N[C]},{\adj}|_{N[C]})$ is a solution to $\Unrooted(N[r],C)$.
\end{enumerate}
For the proof, let $C$ be a component inside $r$.
Let $q\in N(C)$.
If $q\prec r$, then $q\notin N(r)$ and $q$ is not a probe of $(N[S],\prec,\adj)$, so $q\notin N(S)$ and thus $q\in S-N(r)$, which contradicts the assumption that $C$ is a component.
Thus $q\nprec r$.
\Cref{lem:path} applied to the single-edge path between $q$ and the neighbor of $q$ in $C$ yields $q\adj r$, so in particular $q\in N(r)$.
This shows that $N(C)\subseteq N(r)$.
Furthermore, by \ref{item:q-prec}, \ref{item:prec-q}, and \ref{item:adj-q}, $q$ is a probe of $(N[C],{\prec}|_{N[C]},{\adj}|_{N[C]})$.
This shows that $(N[C],{\prec}|_{N[C]},{\adj}|_{N[C]})$ is a solution to $\Unrooted(N[r],C)$.

\begin{enumerate}[resume]
\item\label{item:outer}
For every component $C$ outside $r$, we have $\size{N(C)\cap N(r)}=1$, and $(N[C],{\prec}|_{N[C]},{\adj}|_{N[C]})$ is a solution to $\Rooted(X,q_C,C)$, where $q_C$ is the vertex such that $N(C)\cap N(r)=\set{q_C}$.
\end{enumerate}
For the proof, let $C$ be a component outside $r$.
We have $N(C)\cap N(r)\neq\emptyset$, and it follows from \ref{item:q-prec} and \ref{item:q-adj} that every element of $N(C)\cap N(r)$ is a root in $(N[C],{\prec}|_{N[C]},{\adj}|_{N[C]})$.
Since $N[C]$ is connected, by \Cref{lem:root}, such a root is unique, so $\size{N(C)\cap N(r)}=1$.
Let $q_C$ be the root, so that $N(C)\cap N(r)=\set{q_C}$.
For every $p\in N(C)-\set{q_C}$, we have $p\in N(S)-\set{r}$, so $p$ is a probe.
This shows that $(N[C],{\prec}|_{N[C]},{\adj}|_{N[C]})$ is a solution to $\Rooted(X,q_C,C)$.

\begin{enumerate}[resume]
\item\label{item:probes}
For every $p\in N(S)-\set{r}$, the set $N(p)\cap N[S]$ is contained in the union of $\set{r}$ and the components inside $r$, or is contained in $C\cup\set{q_C}$ for some component $C$ outside $r$, or is equal to $\set{q}$ for some $q\in N(r)\cap S$.
\end{enumerate}
For the proof, let $p\in N(S)-\set{r}$ and $N_p=N(p)\cap N[S]$.
Let $U$ be the union of $\set{r}$, $N(r)\cap S$, and all components outside $r$.
Since $N(S)-\set{r}$ is a set of probes in $(N[S],\prec,\adj)$, every $x\in N_p$ satisfies $p\adj x\notin N(S)-\set{r}$.
In particular, $N_p\subseteq S\cup\set{r}$, so $N_p$ is contained in the union of $U$ and all components inside $r$.
If $N_p\nsubseteq U$, then there is a component $C$ inside $r$ with $N_p\cap C\neq\emptyset$, whence it follows by \ref{item:inner} that $p\in N(r)$ and therefore $r\in N_p\cap U$.
Consequently, $N_p$ satisfies the condition claimed in \ref{item:probes} or $\size{N_p\cap U}\geq 2$.
Suppose the latter.
This and \cref{axiom:2} imply that there are $x,y\in U$ with $p\adj x$, $p\adj y$, and $x\prec y$.
We have $x\neq r$ (as $r$ is the root) and $y\neq r$ (as $x\prec r$ would imply $x\notin U$).
If $x\in N(r)\cap S$, then $x\adj r$, so \cref{axiom:3} entails $r\prec y$, a contradiction.
Thus $x\in C$ for some component $C$ outside $r$.
Since $q_C$ is a root in $(N[C],{\prec}|_{N[C]},{\adj}|_{N[C]})$, we have neither $q_C\prec y$ nor $q_C\adj y$.
By \Cref{lem:path}, a path from $x$ to $q_C$ with all intermediate vertices in $C$ contains a vertex $z$ such that $z\adj y$, where $z\neq y$ by the above.
This shows that $y\in N[C]$, so $y\in C$ or $y=q_C$.
We have thus shown that $N_p\cap U\subseteq C\cup\set{q_C}$ for some component $C$ outside $r$.
Since $N_p\nsubseteq U$ would imply $r\in N_p\cap U$ (as we have shown before), and the latter does not hold, we conclude that $N_p=N_p\cap U\subseteq C\cup\set{q_C}$.

\begin{enumerate}[resume]
\item\label{item:nested}
The family $\famC$ comprising the components inside $r$ is nested.
\end{enumerate}
Let $C\in\famC$ and $x\in C$.
As in the proof of \ref{item:component}, for every $C'\in\famC-\set{C}$, \Cref{lem:path} implies that if $x'\prec x$ for some $x'\in C'$, then $x'\prec x$ for all $x'\in C'$.
In that case, say that $C'$ is \emph{inside} $x$.
Let $C'<C$ denote that $C'$ is inside $x$ for some $x\in C$.
Transitivity of $\prec$ implies transitivity of $<$, which implies that $C_1<C_2$ and $C_2<C_1$ cannot hold simultaneously for any $C_1,C_2\in\famC$.
Now, for the proof of \ref{item:nested}, consider distinct $C_1,C_2\in\famC$.
For every $q\in N(C_1)\cap N(C_2)$, if $x_1$ and $x_2$ are neighbors of $q$ in $C_1$ and $C_2$, respectively, then $q\adj x_1$ and $q\adj x_2$ by \ref{item:adj-q}, which implies $x_1\prec x_2$ or $x_2\prec x_1$ by \cref{axiom:2}.
Consequently, if $N(C_1)\cap N(C_2)\neq\emptyset$, then $C_1<C_2$ or $C_2<C_1$.
Now, consider any $C_1,C_2\in\famC$ with $C_1<C_2$.
Let $q\in N(C_1)$ and $x_1$ be a neighbor of $q$ in $C_1$, so that $q\adj x_1$ by \ref{item:adj-q}.
As we have already shown, if $x_2$ is a neighbor of $q$ in $C_2$, then $x_1\prec x_2$ or $x_2\prec x_1$; the latter would imply $C_2<C_1$, contradicting the assumption that $C_1<C_2$, so $x_1\prec x_2$.
Conversely, by \cref{axiom:4}, if $x_1\prec x_2\in C_2$, then $q\adj x_2$ or $q\prec x_2$; the latter is impossible by \ref{item:q-prec}, so $x_2$ is a neighbor of $q$.
Hence, the neighbors of $q$ in $C_2$ are exactly the vertices $x_2\in C_2$ such that $C_1$ is inside $x_2$, showing $N(C_1)$ is homogeneous for $C_2$.
We conclude that the family $\famC$ is nested.

Now, we use the claims above to show that the algorithm reports a solution to $\Rooted(X,r,S)$.
By \ref{item:inner}, \ref{item:outer}, and the assumption that the smaller subproblems have been correctly solved, every component inside $r$ can be classified as an inner component, and every component outside $r$ can be classified as an outer component.
By \ref{item:component}, every component is either inside or outside $r$, so the algorithm, in \cref{step:rooted-1}, classifies every component as either inner or outer and proceeds to \cref{step:rooted-2}.
If a component $C$ can be classified both ways (where the arbitrary choice made by the algorithm may not agree with being inside or outside $r$), then the conditions $N(C)\subseteq N(r)$ and $\size{N(C)\cap N(r)}=1$ imply $\size{N(C)}=1$, and (since $C$ is a component) the unique element of $N(C)$ belongs to $N(r)\cap S$; in particular, no vertex in $N(S)-\set{r}$ has a neighbor in $C$.
This and \ref{item:probes} imply that the algorithm verifies the condition in \cref{step:rooted-2} positively and proceeds to \cref{step:rooted-3}.
By \ref{item:nested} and the fact that a nested family remains nested after removing some set or adding a set $C$ with $\size{N(C)}=1$, the family of all inner components is nested.
Therefore, the algorithm proceeds to \cref{step:rooted-4}, where it computes (correctly) some solution to $\Rooted(X,r,S)$.
\end{proof}

The graph $G$ is a Burling graph if and only if the subproblem $\Unrooted(\emptyset,C)$, which asks for a Burling structure $(C,\prec_C,\adj_C)$, has a solution for every component $C$ of $V$.
From \Cref{lem:unrooted,lem:rooted}, by induction on the size of the subproblems, we conclude that the algorithm solves the subproblem $\Unrooted(\emptyset,C)$ correctly for each component $C$ of $V$.
If it finds a solution $(C,\prec_C,\adj_C)$ for each component $C$ of $V$ (rather than reporting no solution), a Burling set $(V,\prec,\adj)$ with $G=(V,\edgeset{\adj})$ is obtained by defining $\prec$ and $\adj$ as the unions of $\prec_C$ and $\adj_C$ (respectively) over all components $C$ of $V$.

It is clear that the algorithm to solve a subproblem of the form $\Unrooted(X,S)$ or $\Rooted(X,r,S)$ runs in polynomial time assuming that solutions to smaller subproblems have been computed beforehand.
Since the number of subproblems is polynomial, the total running time to solve $\Unrooted(\emptyset,C)$ on all components $C$ of $V$ is polynomial.
Since checking whether $G$ is triangle-free can also be done in polynomial time, we arrive at the following conclusion.

\begin{theorem}
\label{thm:recognition-abstract}
There is a polynomial-time algorithm that, given a graph\/ $G$, either declares\/ $G$ a Burling graph and produces a Burling set\/ $(V,\prec,\adj)$ such that\/ $G=(V,\edgeset{\adj})$, or reports that\/ $G$ is not a Burling graph.
\end{theorem}

\section{Strict frame representation}
\label{sec:construction}

In this section, we provide a polynomial-time algorithm to turn a Burling set into a strict frame representation of it.
The existence of such a representation, claimed in \Cref{thm:characterization}, was proved by Pournajafi and Trotignon~\cite{PT23}, but that proof does not produce the representation explicitly.
Instead, it relies on the existence of a strict frame representation (constructed in~\cite{PKK+13}) of a graph $B_k$ containing a given Burling graph $G$ as an induced subgraph, and the size of that graph $B_k$ may be double exponential in the size of $G$.

We fix a Burling set $(S,\prec,\adj)$ for which we want to construct a strict frame representation $\set{F_x}_{x\in S}$ with $F_x=F(\ell_x,r_x,b_x,t_x)$ for all $x\in S$.
Let ${\rel}={\prec}\cup{\adj}$.
By \Cref{lem:acyclic}, $\rel$ is an acyclic relation on $S$.
Let $\rel^*$ be the transitive closure of $\rel$, that is, for all $x,y\in S$, we have $x\rel^*y$ if and only if $x=x_0\rel\cdots\rel x_k=y$ for some $x_0,\ldots,x_k\in S$ with $k\geq 1$.
Since $\rel$ is acyclic, $\rel^*$ is also acyclic, so it is a strict partial order.

We first show how to find appropriate values $\ell_x$ and $r_x$ for all $x\in S$.
Let $\rel\adj$ be the relational composition of $\rel$ and $\adj$, that is, for all $x,z\in S$, we have $x\rel\adj z$ if and only if there is $y\in S$ such that $x\rel y\adj z$.
Consider the following conditions on a set of $2\size{S}$ distinct symbols $\ell_x$ and $r_x$ with $x\in S$:
\begin{conditions}
\item\label{cond:1} $\alignto{r_x}{r}{\ell_x}<r_x$ for all $x\in S$;
\item\label{cond:2} $\alignto{r_x}{r}{\ell_x}<\alignto{r_x}{l}{\ell_y}$ for all $x,y\in S$ with $y\rel x$;
\item\label{cond:3} $r_x<\alignto{r_x}{l}{r_y}$ for all $x,y\in S$ with $x\prec y$ or $y\adj x$;
\item\label{cond:4} $r_x<\alignto{r_x}{l}{\ell_y}$ for all $x,y\in S$ with $y\rel\adj x$.
\end{conditions}

\begin{lemma}
\label{lem:horizontal}
The minimal binary relation\/ $<$ defined by \crefrange{cond:1}{cond:4} on the set\/ $\set{\ell_x,r_x}_{x\in S}$ is acyclic.
\end{lemma}

\begin{proof}
Suppose not.
Let $C$ be a smallest cycle in $<$.
For any $c\in C$, call an element $c'\in C$ a \emph{predecessor} or a \emph{successor} of $c$ if $c'<c$ or $c<c'$, respectively.
From all elements $x\in S$ with $C\cap\set{\ell_x,r_x}\neq\emptyset$, choose one that is minimal in the order $\rel^*$.
Thus $C\cap\set{\ell_y,r_y}=\emptyset$ for all $y\in S$ with $y\rel^*x$.
This and \crefrange{cond:1}{cond:4} from the definition of $<$ imply that
\begin{itemize}
\item if $\ell_x\in C$, then $r_x$ is the successor of $\ell_x$; in particular, $r_x\in C$;
\item the successor of $r_x$ has form $r_y$ for some $y\in S$ with $x\prec y$;
\item the predecessor of $r_x$ is $\ell_x$ or has form $r_z$ for some $z\in S$ with $x\adj z$;
\item if $\ell_x\in C$, then the predecessor of $\ell_x$ has form $\ell_z$ for some $z\in S$ with $x\rel z$, or $r_z$ for some $z\in S$ with $x\rel\adj z$.
\end{itemize}
This leaves the following options for how the cycle $C$ looks around its intersection with $\set{\ell_x,r_x}$.
\begin{options}
\item $r_z<r_x<r_y$ for some $y,z\in S$ with $x\prec y$ and $x\adj z$.
\Cref{axiom:3} entails $z\prec y$, which implies $r_z<r_y$ by \labelcref{cond:3}, showing that $C-\set{r_x}$ is a cycle in $<$.
\item $\ell_z<\ell_x<r_x<r_y$ for some $y,z\in S$ with $x\prec y$ and $x\rel z$.
\Cref{axiom:1,axiom:3} entail $y=z$, or $y\prec z$, or $z\prec y$.
Now, if $y=z$, then $\ell_z<r_y$ by \labelcref{cond:1}, showing that $C-\set{\ell_x,r_x}$ is a cycle in $<$.
If $y\prec z$, then $\ell_z<\ell_y<r_y$ by \labelcref{cond:2,cond:1}, showing that $C-\set{\ell_x,r_x}\cup\set{\ell_y}$ is a cycle in $<$.
If $z\prec y$, then $\ell_z<r_z<r_y$ by \labelcref{cond:1,cond:3}, showing that $C-\set{\ell_x,r_x}\cup\set{r_z}$ is a cycle in $<$.
\item $r_z<\ell_x<r_x<r_y$ for some $y,z\in S$ with $x\prec y$ and $x\rel\adj z$.
There is $z'\in S$ with $x\rel z'\adj z$.
Since $x\prec y$ and $x\rel z'$, \cref{axiom:1,axiom:3} entail $y=z'$, or $y\prec z'$, or $z'\prec y$, which imply, respectively, $y\adj z$, or $y\rel\adj z$, or $z\prec y$ by \cref{axiom:3}.
Now, if $z\prec y$ or $y\adj z$, then $r_z<r_y$ by \labelcref{cond:3}, showing that $C-\set{\ell_x,r_x}$ is a cycle in $<$.
If $y\rel\adj z$, then $r_z<\ell_y<r_y$ by \labelcref{cond:4,cond:1}, showing that $C-\set{\ell_x,r_x}\cup\set{\ell_y}$ is a cycle in $<$.
\end{options}
Each option leads to a cycle in $<$ smaller than $C$, contradicting the assumption that $C$ is a smallest cycle in $<$.
This shows that the relation $<$ is indeed acyclic.
\end{proof}

By \Cref{lem:horizontal}, the symbols $\ell_x$ and $r_x$ with $x\in S$ can be assigned values $1,\ldots,2\size{S}$ so that \crefrange{cond:1}{cond:4} hold for the assigned values.

Now, we describe how to find appropriate values $b_x$ and $t_x$ for all $x\in S$.
A \emph{parent} of an element $x\in S$ is an element $z\in S$ such that $x\rel z$ and there is no $y\in S$ with $x\rel y\rel z$.
By \crefrange{axiom:1}{axiom:3}, every element of $S$ that is not a root of $(S,\prec,\adj)$ has a unique parent.
Therefore, the set $S$ along with the parent relation forms a (directed) forest such that if $x\rel^*y$, then $y$ is an ancestor of $x$.
To determine $b_x$ and $t_x$ for all $x\in S$, perform a depth-first search on this forest starting from each of the roots (in some order), and record, for each $x\in S$, the discovery time $b_x$ and the finishing time $t_x$.
This leads to an assignment of values $1,\ldots,2\size{S}$ such that
\begin{conditions}[resume]
\item\label{cond:5} $b_x<t_x$ for all $x\in S$;
\item\label{cond:6} $\alignto{b_x}{l}{b_y}<b_x<t_x<t_y$ for all $x,y\in S$ with $x\rel^*y$;
\item\label{cond:7} $\alignto{b_x}{r}{t_x}<b_y$ or $t_y<b_x$ for all $x,y\in S$ with neither $x\rel^*y$ nor $y\rel^*x$.
\end{conditions}

\begin{lemma}
\label{lem:frames}
The family\/ $\set{F(\ell_x,r_x,b_x,t_x)}_{x\in S}$ is a strict frame representation of\/ $(S,\prec,\adj)$.
\end{lemma}

\begin{proof}
\Cref{cond:1,cond:5} ensure that $F(\ell_x,r_x,b_x,t_x)$ is a frame for every $x\in S$.
To prove that the family $\set{F(\ell_x,r_x,b_x,t_x)}_{x\in S}$ is a strict frame representation of $(S,\prec,\adj)$, we first verify the following three conditions for all $x,y\in S$, which imply the conditions from the definition of a strict frame representation of a Burling set:
\begin{alignat*}{2}
&\text{if}\enspace x\prec y{,}\quad&&\text{then}\enspace\ell_y<\ell_x<r_x<r_y\enspace\text{and}\enspace b_y<b_x<t_x<t_y{;}\\
&\text{if}\enspace x\adj y{,}\quad&&\text{then}\enspace\ell_y<\ell_x<r_y<r_x\enspace\text{and}\enspace b_y<b_x<t_x<t_y{;}\\
&\text{if neither}\enspace x\rel y\enspace\text{nor}\enspace y\rel x{,}\quad&&\text{then}\enspace r_x<\ell_y{,}\enspace\text{or}\enspace r_y<\ell_x{,}\enspace\text{or}\enspace t_x<b_y{,}\enspace\text{or}\enspace t_y<b_x{.}
\end{alignat*}
The first two conditions describe the configurations in \Cref{fig:correspondence}, while the last one describes two frames that are disjoint and not nested.

If $x\prec y$, then $\ell_y<\ell_x<r_x<r_y$ by \labelcref{cond:2,cond:3}, and $b_y<b_x<t_x<t_y$ by \labelcref{cond:6}.
Likewise, if $x\adj y$, then $\ell_y<\ell_x<r_y<r_x$ by \labelcref{cond:2,cond:3}, and $b_y<b_x<t_x<t_y$ by \labelcref{cond:6}.
If neither $x\rel^*y$ nor $y\rel^*x$, then $t_x<b_y$ or $t_y<b_x$ by \labelcref{cond:7}.
Finally, suppose $x\rel^*y$ but not $x\rel y$.
Let $x_0,\ldots,x_k\in S$ be a smallest tuple with $x=x_0\rel\cdots\rel x_k=y$ witnessing $x\rel^*y$.
It follows that $k\geq 2$.
If $x_i\prec x_{i+1}$ for some $i\in\set{1,\ldots,k-1}$, then $x_{i-1}\rel x_{i+1}$ by transitivity of $\prec$ or by \cref{axiom:4}, so $x_0,\ldots,x_{i-1},x_{i+1},\ldots,x_k$ is a smaller tuple witnessing $x\rel^*y$, contradicting the assumption that $x_0,\ldots,x_k$ is smallest.
Therefore, we have $x\rel x_1\adj\cdots\adj x_k=y$, whence it follows that $r_y=r_{x_k}<\cdots<r_{x_2}<\ell_x$.

It remains to verify that the family $\set{F(\ell_x,r_x,b_x,t_x)}_{x\in S}$ is strict.
By the above, a configuration of three frames in \Cref{fig:strict} (b)--(d) would correspond to three elements $x,y,z\in S$ with $x\rel y$, $x\rel z$, and $y\adj z$.
However, $x\rel y$, $x\rel z$, and $y\neq z$ imply $y\prec z$ or $z\prec y$ by \crefrange{axiom:1}{axiom:3}.
\end{proof}

The following is a direct corollary to the construction above and \Cref{lem:frames}.

\begin{theorem}
\label{thm:construction}
There is a polynomial-time algorithm that constructs a strict frame representation for a given Burling set.
\end{theorem}

The algorithm can be made to work in $\Oh(\size{S}+\size{\rel})$ time by relaxing \cref{cond:4} only to at most $\size{\rel}$ pairs $x,y\in S$ with $y\rel\adj x$.
Specifically, for all $y,z\in S$ with $y\rel z$, \cref{axiom:2} and acyclicity of $\prec$ imply that the elements $x\in S$ with $z\adj x$ are totally ordered by $\prec$, and the requirement that $r_x<\ell_y$ if $y\rel\adj x$ is redundant for all but the maximal one in that order.
Indeed, if $x'\prec x$, then $r_{x'}<r_x$ by \labelcref{cond:3}, so $r_x<\ell_y$ implies $r_{x'}<\ell_y$.
With that modification, the relation $<$ defined by \crefrange{cond:1}{cond:4} has size $\Oh(\size{S}+\size{R})$.
Its topological sort used to determine the values of $\ell_x$ and $r_x$ as well as the depth-first search used to determine the values of $b_x$ and $t_x$ clearly work in linear time.

\Cref{thm:recognition-frames} follows directly from \Cref{thm:recognition-abstract,thm:construction}.

\section{Maximum independent set}
\label{sec:mis}

In this section, we present an algorithm for the maximum (weight) independent set problem in Burling graphs.
The core of the algorithm is a dynamic programming scheme developed by Gavril \cite{Gav73,Gav00} (see also~\cite{CS03}) which reduces, for any class of families of geometric objects $\famF$, the maximum (weight) independent set problem in the \emph{overlap graphs} of families in $\famF$ to the maximum weight independent set problem in the \emph{intersection graphs} of families in $\famF$.
In our case, $\famF$ is the class of strict families of axis-parallel rectangles (defined just like strict families of frames except that rectangles include their interiors while frames do not).
Strict frame graphs are the overlap graphs of families in $\famF$, while the intersection graphs of families in $\famF$ turn out to be chordal.

Say that a relation $\rel$ on a set $S$ is \emph{chordal} when it is acyclic and the following holds for all $x,y,z\in S$: if $x\rel y$, and $x\rel z$, and $y\neq z$, then $y\rel z$ or $z\rel y$.
This term relates to well-known properties of chordal graphs.
In particular, $(S,\edgeset{\rel})$ is a chordal graph when $\rel$ is a chordal relation on $S$, and vice versa---for every chordal graph $G$, orienting every edge according to a so-called perfect elimination order gives rise to a chordal relation on the vertex set.
Apart from the next theorem, we will not use any particular properties of chordal graphs, so we omit further details.

\begin{theorem}[Frank~\cite{Fra76}]
\label{thm:chordal}
There is a linear-time algorithm that, given a non-empty finite set\/ $S$, a chordal relation\/ $\rel$ on\/ $S$, and a weight assignment\/ $\weight\colon S\to\setR$, computes a maximum independent set in the graph\/ $(S,\edgeset{\rel})$ with respect to the weight assignment\/ $\weight$.
\end{theorem}

Let $(V,\prec,\adj)$ be a Burling set, and let $G=(V,\edgeset{\adj})$.
Let $\weight\colon V\to\setR$ be a weight assignment according to which we want to compute a maximum independent set in $G$.
For a set $S\subseteq V$, let $\weight(S)=\sum_{x\in S}\weight(x)$.
Let ${\rel}={\prec}\cup{\adj}$.
\Cref{lem:acyclic} and \crefrange{axiom:1}{axiom:3} imply that the relation $\rel$ is chordal.
Let $V_u=\set{x\in V\colon x\prec u}$ for all $u\in V$.
Note that if $x\in V_u$, then $V_x\subset V_u$.
The algorithm applies the dynamic programming technique on subproblems of the form $\Indep(S)$ defined for $S=V$ and $S=V_u$ for all $u\in V$, where the subproblem $\Indep(S)$ asks for a maximum independent set in the graph $(S,\edgeset{{\adj}|_S})$ with respect to the weight assignment $\weight$.
The \emph{size} of such a subproblem is the size of $S$.
A \emph{solution} to a subproblem is a maximum independent set being asked for.

\begin{algorithm}{$\Indep(S)$}
\item If $S=\emptyset$, then stop and report the empty set as a solution, otherwise continue.
\item For every $u\in S$, let $I_u$ be a solution to the subproblem $\Indep(V_u)$, and let $\weight^*(u)=\weight(u)+\weight(I_u)$.
\item Compute a maximum independent set $I^*$ in the graph $(S,\edgeset{{\rel}|_S})$ with weight assignment $\weight^*$, using the algorithm claimed in \Cref{thm:chordal}.
\item Report the set $I^*\cup\bigcup_{u\in I^*}I_u$ as a solution to $\Indep(S)$.
\end{algorithm}

\begin{lemma}[cf.\ {\cite[Lemma~2]{Gav00}}]
\label{lem:independent}
The algorithm to solve\/ $\Indep(S)$ is correct assuming that the smaller subproblems have been correctly solved.
\end{lemma}

\begin{proof}
Let $I$ be the set reported by the algorithm as a solution to $\Indep(S)$.
Thus $I=\bigcup_{u\in I^*}(I_u\cup\set{u})$.
For every $u\in S$, since $I_u\subseteq V_u=\set{x\in V\colon x\prec u}\subseteq S$ and the set $I_u$ is independent in $(V,\edgeset{\adj})$, so is the set $I_u\cup\set{u}$.
Suppose $I_u\cup\set{u}\ni x\adj y\in I_v\cup\set{v}$ for some distinct $u,v\in I^*$.
Since $x\adj y$ and $y=v$ or $y\prec v$, \cref{axiom:4} entails $x\rel v$.
Since $x=u$ or $x\prec u$, $x\rel v$, and $u\neq v$, \cref{axiom:1} or \labelcref{axiom:3} entails $u\rel v$ or $v\rel u$, which contradicts the assumption that $I^*$ is independent in $(S,\edgeset{{\rel}|_S})$.
This shows that $I$ is an independent set in $(S,\edgeset{{\adj}|_S})$.

Now, let $J$ be an arbitrary independent set in $(S,\edgeset{{\adj}|_S})$.
It follows that $\weight(J\cap V_u)\leq\weight(I_u)$ for all $u\in S$.
Let $J^*=\{u\in J\colon$there is no $v\in J$ with $u\prec v\}$.
It follows that $J^*$ is an independent set in $(S,\edgeset{{\rel}|_S})$, so $\weight^*(J^*)\leq\weight^*(I^*)$.
Since $J=J^*\cup\bigcup_{u\in J^*}(J\cap V_u)$, we conclude that
\[\weight(J)\leq\weight(J^*)+\sum_{u\in J^*}\weight(J\cap V_u)\leq\sum_{u\in J^*}\weight(u)+\sum_{u\in J^*}\weight(I_u)=\sum_{u\in J^*}\weight^*(u)=\weight^*(J^*)\leq\weight^*(I^*){.}\]
This shows that $I$ is a maximum independent set in $(S,\edgeset{{\adj}|_S})$ with respect to the weight assignment $\weight$.
\end{proof}

The subproblem $\Indep(V)$ asks for a maximum independent set in the graph $(V,\edgeset{\adj})$.
It is clear that the algorithm to solve a subproblem of the form $\Indep(S)$ runs in polynomial time assuming that solutions to smaller subproblems have been computed beforehand.
Since there are $\size{V}+1$ subproblems, the total running time spent to solve $\Indep(V)$ is polynomial.
Hence, we arrive at the following conclusion.

\begin{theorem}
\label{thm:mis-abstract}
There is a polynomial-time algorithm that, given a Burling set\/ $(V,\prec,\adj)$ and a weight assignment\/ $\weight\colon V\to\setR$, computes a maximum independent set in the graph\/ $(V,\edgeset{\adj})$ with respect to the weight assignment\/ $\weight$.
\end{theorem}

\Cref{thm:mis-burling} follows directly from \Cref{thm:recognition-abstract,thm:mis-abstract}.

\section{Concluding Remarks}
\label{sec:conclusion}

Intersection graphs of axis-parallel rectangles in the plane are \classNP-hard to recognize~\cite{Kra94}.
An easy adaptation of the proof in~\cite{Kra94} shows that frame graphs, i.e., graphs admitting some (not necessarily strict) frame representation, are also \classNP-hard to recognize.
The maximum independent set problem is \classNP-hard for intersection graphs of axis-parallel unit squares~\cite{FPT81} and for intersection graphs of horizontal and vertical segments~\cite{KN90}, either of which directly implies \classNP-hardness for frame graphs.
The situation is different for \emph{directed frame graphs}, i.e., graphs with frame representations in which every intersecting pair of frames looks like in \Cref{fig:strict}~(a) (which is the first condition from the definition of \emph{strict}).
For directed frame graphs, the relation $\adj$ defined as in \Cref{fig:correspondence} is chordal, and therefore the maximum independent set problem can be solved in polynomial time in the same way as described in \Cref{sec:mis} provided that the corresponding relations $\prec$ and $\adj$ are given as part of the input.
This leaves the following question.

\begin{problem}
What is the complexity of recognizing directed frame graphs?
\end{problem}

Burling graphs are induced subgraphs of Burling's construction of triangle-free high-chromatic graphs.
What is the complexity of recognition and of computing a maximum independent set for other known constructions of triangle-free high-chromatic graphs?
This is not interesting for Mycielski's construction~\cite{Myc55}, which contains all triangle-free graphs as induced subgraphs~\cite{CGL06}.
For induced subgraphs of Zykov's construction~\cite{Zyk49} and Blanche Descartes' construction~\cite{Des54}, both problems are \classNP-hard~\cite{MTTW26}.
For twin-cut graphs, i.e., induced subgraphs of the construction in~\cite{BBD+23}, recognition is in \classP\ while computing a maximum independent set is \classNP-hard~\cite{BBD+unpub}.

We will show in a follow-up work that, by contrast to the maximum independent set problem, the $k$-coloring problem remains \classNP-hard for Burling graphs for every $k\geq 3$, as it is for the other above-mentioned classes~\cite{BBD+unpub,MTTW26}.
This leads to the following analog of the Thomassé--Trotignon--Vušković question from~\cite{TTV17}.

\begin{problem}
Is every hereditary class of graphs that admits a polynomial-time $3$-coloring algorithm $\chi$-bounded (assuming $\classP\neq\classNP$)?
\end{problem}

As we have mentioned in the introduction, the recent flurry of research on Burling graphs has been mostly motivated by the conjecture of Scott~\cite{Sco97} that the class $\famF(H)$ of graphs excluding induced subdivisions of $H$ is $\chi$-bounded for every graph $H$.
Clearly, if no subdivision of $H$ is a Burling graph, then Scott's conjecture fails for $H$, because $\famF(H)$ contains all graphs $B_1,B_2,\ldots$ from Burling's construction.
Chudnovsky, Scott, and Seymour~\cite{CSS21} conjectured that for every $H$, the chromatic number of graphs $G\in\famF(H)$ is bounded by a function of the clique number of $G$ and the maximum $k$ such that $G$ contains an induced subgraph isomorphic to $B_k$.
This conjecture, if true, implies that if some subdivision of $H$ is a Burling graph, then Scott's conjecture holds for $H$.
Our recognition algorithm for Burling graphs can be adapted to a polynomial-time recognition algorithm of graphs $H$ such that some subdivision of $H$ is a Burling graph.
We omit the details, as this direction seems less compelling while the Chudnovsky--Scott--Seymour conjecture remains open.
However, very recently, the special case of the conjecture for graphs $G$ that are string graphs has been proved~\cite{ABD+26}.
In particular, we now know that Burling graphs form a minimal hereditary class of graphs with unbounded chromatic number---the only known class of graphs with that property other than the (trivial) class of complete graphs.

\section*{Acknowledgments}

We thank Martin Milanič for bringing the question of Thomassé, Trotignon, and Vušković to our attention.
We thank Jonathan Rollin and the anonymous reviewers for helpful comments.

The project was initiated at the Structural Graph Theory workshop STWOR in September 2023, which was a part of STRUG: Structural Graph Theory Bootcamp, funded by the ``Excellence initiative -- research university (2020--2026)'' of University of Warsaw.
The main part of the work was done during the workshop Homonolo 2023.
We thank the participants and the organizers of both events for a productive and inspiring atmosphere.

\bibliographystyle{plain}
\bibliography{burling}

\end{document}